\documentclass[12pt]{amsart}
\usepackage{amsmath}
\usepackage{amsfonts}
\usepackage{amssymb}
\pdfoutput=1
\usepackage{hyperref}

\DeclareMathOperator{\sgn}{sgn}

\DeclareMathOperator{\Real}{Re}

\DeclareMathOperator{\R}{\mathbb{R}}

\newtheorem{theorem}{Theorem}[section]
\newtheorem{lemma}[theorem]{Lemma}

\theoremstyle{definition}

\theoremstyle{remark}
\newtheorem{remark}[theorem]{Remark}

\numberwithin{equation}{section}

\begin{document}

\def\now{23 May 2021}

\title[\now \hfill 
Fourier majorants
\hfill]
{
Fourier majorants
that match norms}
\author
[Fournier
]
{John J. F. Fournier
}
\address{Department of Mathematics\\
University of British Columbia\\exponents
Vancouver, Canada V6T 1Z2}
\email{fournier@math.ubc.ca}
\thanks{Results partially announced at the June 2001 meeting
of the Canadian Mathematical Society.}
\thanks{
Research partially
supported by NSERC grant 4822.}

\date{\now}

\begin{abstract}
Denote the
coefficients in the complex form of the  Fourier series of a
function~$f$ on the interval~$[-\pi, \pi)$ by~$\hat f(n)$.
It is known that if~$p = 2j/(2j-1)$ for some
integer~$j>0$, then for each function~$f$ in~$L^p$ there exists 
another
function $F$ in~$L^p$ that majorizes~$f$ in the sense that~$\hat
F(n) \ge |\hat f(n)|$ for all~$n$, but that also satisfies~$\|F\|_p
\le \|f\|_p$.
Rescaling~$F$ suitably then gives a majorant with the same~$L^p$ norm as~$f$.
We show how
that
majorant
comes from a variant in~$L^{2j}$ of the notion of exact majorant in~$L^2$.
\end{abstract}

\maketitle

\markleft{\hfill John J.F. Fournier\hfill \now}

\section{Introduction}
\label{sec:intro}

{\allowdisplaybreaks}

Thus~$\hat f(n) = \int_{-\pi}^\pi
f(\theta)e^{-in\theta}\,d\theta/2\pi$. Call~$F$ a {\em majorant}
of~$f$,
and~$f$ a {\em minorant}
of~$F$,
when~$|\hat f(n)| \le \hat F(n)$
for all integers~$n$. In that case,~$\|f\|_2 \le \|F\|_2$; also, if~$j$ is
an integer greater than~$1$, then $F^j$ majorizes~$f^j$, and hence
\[(\|f\|_{2j})^{2j} = (\|f^j\|_2)^2 \le (\|F^j\|_2)^2 =
(\|F\|_{2j})^{2j}.\]
Finally,~$\|f\|_\infty \le \|F\|_\infty$ when~$F$ majorizes~$f$.

This pattern does not persist for other 
exponents.
Hardy and Littlewood~\cite{HaL}
considered the {\em upper majorant property},
asserting that there is a
constant~$U(p)$ so that
\begin{equation}
\label{eq:upper}
\|f\|_p \le U(p) \|F\|_p
\end{equation}
whenever~$F$ majorizes~$f$.
They gave an example
showing that if
this property
holds in~$L^3$ then the constant~$U(3)$
must be strictly larger than~$1$.
Other
work~\cite{Bo, Ba}
revealed
that the property fails
for 
the exponents~$p$
in the interval~$(0, \infty)$
that are not
even integers.
See~\cite[pp.~131--134]{Mt}, \cite{GR}, \cite{Gr},\cite{MS},~\cite{CKY} and~\cite{Ebe}
for refinements
of these results, 
complements to them
and 
connections with other 
questions.

Here we consider the {\em lower majorant property},
also introduced in~\cite{HaL}.
It holds when there is a constant~$L(p)$ so that
each function $f$ in $L^p$
has a majorant $F$ with
\begin{equation}
\label{eq:lowmajor}
\|F\|_p \le L(p)\|f\|_p.
\end{equation}
This is clearly true when~$p=2$ with~$L(2) = 1$, with~$F$ equal to
the {\em exact majorant} of~$f$ given by~$\hat F = |\hat f|$.
It also holds when~$p = 1$,
since one can
write
a given function~$f$
in~$L^1$
as the square of a function~$g$ in~$L^2$,
form the exact
majorant~$G$ of~$g$,
and
let~$F = G^2$.
Then
\begin{gather}
\notag
\|F\|_1 =  \frac{1}{2\pi}\int_{-\pi}^\pi
|G(t)|^2\,dt
= (\|G\|_2)^2
\\
\notag
=(\|g\|_2)^2
= \frac{1}{2\pi}\int_{-\pi}^\pi
|g(t)|^2\,dt
= \|f\|_1.
\end{gather}
The sequence~$\hat F$ is equal to the convolution of two copies of the sequence~$\hat G$, and similarly for~$\hat f$ and~$\hat g$.
Since there is no cancelation in the convolution giving~$F$,
the
fact that~$\hat G
=
|\hat g|$
implies that~$\hat F \ge |\hat f|$.

When~$p \in [1, \infty)$, a simple duality argument \cite[Section 3]{Bo}
shows that if~$L^p$ has the lower majorant property,
then its dual space~$L^{p'}$ has the
upper majorant property with~$U(p') \le L(p)$. By the
work cited 
earlier,
this can only happen when~$p' = \infty$ or~$p'$
is an even integer, thus ruling out all exponents~$p$ in the
interval~$(1, \infty)$ except for~$p = 2$ and the
{\em 
special
exponents}
with~$p=2j/(2j-1)$, where~$j$ is an integer strictly greater
than~$1$. 
When~$1 < p < \infty$, a
less simple
duality argument~\cite{HaL, Ba} shows
that the upper majorant property for~$L^{p'}$ implies the lower
majorant property for~$L^p$,
with~$L(p) \le U(p')$.
In particular,
it holds,
with~$L(p)=1$,
for the special exponents.

That duality proof and alternatives~\cite{Ob, DGLPQ, LuP} to it
do not include a
description
of a suitable majorant of a given function for these values of~$p$.
For various good reasons, those arguments covered
general exponents~$p$ and~$p'$, 
but
the duality is now known to mainly have impact for the 
special
values of~$p$.

In~\cite{FV}, this hindsight was exploited by
showing in
those
cases
that a 
certain
product of~$2j-1$ functions in~$L^{2j}$ gives
a majorant with
minimal norm in~$L^{2j/(2j-1)}$.
Suitably rescaling
that product
yields
a majorant with the same~$L^{2j/(2j-1)}$ norm as the original function.

In this paper, we offer a more direct description of
that
particular
majorant.
We state our conclusions
about it
in the next section,
and prove those conclusions
in the following section.
Then we
reformulate
them as statements
about convolution on the integers, and comment on how
they extend
to all discrete groups, abelian or not.
In
an appendix,
we discuss
the modifications needed on groups like~$\R^n$.

\section{Dual maximizers}
\label{sec:Maximizers}

We suppose in the rest of this paper that~$1 < p < \infty$.
In this paragraph, we summarize
some of
the reasoning
in~\cite{HaL}.
First, if~$f \in L^p$ and~$g \in L^{p'}$, then the
series~$\sum_n \hat f(n)\overline{\hat g(n)}$
converges to$\int_{-\pi}^\pi f(t)\overline{g(t)}dt/2\pi$.
By H\"older's inequality,
\begin{equation}
\label{eq:Parseval}
\left|\sum_n \hat f(n)\overline{\hat g(n)}\right|
\le \|f\|_p\|g\|_{p'}.
\end{equation}
In the special case where~$p'$ is even and $\hat G \ge 0$, it is also true that
\begin{equation}
\label{eq:LinearizedUpper}
\sum_n |\hat f(n)|\hat G(n) \le \|f\|_p\|G\|_{p'},
\end{equation}
because each term~$|\hat f(n)|\hat G(n)$
can be rewritten as~$\hat f(n)\overline{\hat g(n)}$
for a function~$g$ that is majorized by~$G$.

Recall
that
a 
complex function-valued function~$f$ factors as~$|f|\sgn(f)$, where~$\sgn(f)$ vanishes off the support of~$f$.
If~$f \in L^p$, then
letting~$g = |f|^{p-1}\sgn(f)$
puts~$g$ in ~$L^{p'}$, and makes~$f = |g|^{p'-1}\sgn(g)$.

In the special cases where~$p' = 2j$ for some positive integer~$j$, this simplifies to
\begin{gather}
\notag
f = |g|^{2j-1}\sgn(g) 
= |g|^{2j-2}|g|\sgn(g)
\\
\notag
=|g|^{2(j-1)}g = (g{\overline g})^{j-1}g.
\end{gather}
Moreover,~$|f|^{2j/(2j-1)} = |g|^{2j}$,
and~$\|f\|_{2j/(2j-1)}= \|g\|_{2j}^{2j-1}$.

Denote the set of functions~$G$ in~$L^{2j}$ for which~$\hat G \ge 0$
by~$PD^{2j}$.
Also consider
its closed subset~$PD_f^{2j}$
where~$\hat G$
vanishes off the support of~$\hat f$,
and the closed subset~$PD^{2j}_{(g, f)}$
of~$PD_f^{2j}$
where~$G$
has the same norm in~$L^{2j}$
as~$g$.
%
Let~$\Phi_f$
be the function on~$PD^{2j}$ that maps~$G$ 
to~$\sum_n |\hat f(n)|\hat G(n)$,
which is finite by inequality~\eqref{eq:LinearizedUpper}.

\begin{theorem}
\label{th:Maximizer}

There
is a unique function~$G$ in the set~$PD^{2j}(f, g)$ for which~$\Phi_f(G)$ is maximal in that
set.
The
product~$F := (G\overline G)^{j-1}G$
then has the same norm in~$L^{2j/(2j-1)}$ as~$f$,
and~$F$ majorizes~$f$.

\end{theorem}

\begin{remark}
\label{rm:JustExact}

When~$j=1$, the two functions~$f$ and~$g$ coincide, as do~$G$ and~$F$,
which
are
their
exact
majorants.
When~$j > 1$,
it may be reasonable to say that the pair~$(G, F)$ specified above is the \emph{exact majorant in~$L^{2j}\times L^{2j/(2j-1)}$} of the pair~$(g, f)$.

\end{remark}

\begin{remark}
\label{Sidon}

As in~\cite{Ba}, the function~$f$ above has a unique
majorant,~$F'$ say,
of minimal~$L^{2j/(2j-1)}$ norm.
That majorant coincides with the
one
in Theorem~\ref{th:Maximizer}
if and only if~$\|F'\|_{2j/(2j-1)} = \|f\|_{2j/(2j-1)}$.
It
is
shown
in~\cite[\S5]{FV}
that those two norms agree
if and only if forming the exact majorant in~$L^2$
of the function~$g$ above does not change its~$L^{2j}$ norm.
\end{remark}

\begin{remark}

These are the cases where the
maximizing
function~$G$
above
is the exact majorant of~$g$ in~$L^2$.
In
particular,~$\hat G = |\hat g|$
if the support of~$\hat g$ is a Sidon~$B_{j}$ set,
that is each number in the sum of~$j$ copies of the set has a unique representation modulo permutation of the~$j$ summands.


\end{remark}

\begin{remark}
\label{rm:Promise}

We
will comment
at end of the next
section on the
cases
where~$\|F'\|_{2j/(2j-1)} < \|f\|_{2j/(2j-1)}$.

\end{remark}

\section{More details}
\label{sec:Details}

\begin
{proof}[Proof of Theorem~\ref{th:Maximizer}.]

Let~$L^{2j}_f$
comprise all functions~$h$ in~$L^{2j}$
whose coefficients vanish
off the support of~$\hat f$.
When~$f$ is a trigonometric polynomial,~$L^{2j}_f$ is finite dimensional, and its subset~$PD^{2j}_{(g,f)}$ is compact.
Now~$\Phi_f$ is obviously continuous on~$PD^{2j}_{(g,f)}$ in this case; so there is a function~$G$ in that set at which~$\Phi_f(G)$ is maximal.
The existence
of such a maximizer~$G$
for other functions~$f$ will be established at the end of this proof.

Let~$M = \Phi_f(G)$.
If~$\Phi_f(H) = M$ for
another function~$H$ 
in~$PD^{2j}_{(g, f)}$,
let~$G' = (G+H)/2$.  Then~$\Phi_f(G') = M$
too,
but~$\|G'\|_{2j} < \|g\|_{2j}$
by the strict convexity of the~$L^{2j}$ norm.
Rescaling~$G'$
to have the same~$L^{2j}$ norm as~$g$
would then give a function in~$PD^{2j}(g,f)$ where~$\Phi_f$ takes a value larger than~$M$, contrary to the maximality of~$M$.
So~$G$ is unique.

As noted earlier,~$\|f\|_{2j/(2j-1)}= \|g\|_{2j}^{2j-1}$.
The corresponding statement about~$F$ and~$G$ holds for the same reasons.
It then follows
that~$\|F\|_{2j/(2j-1)} = \|f\|_{2j(2j-1)}$
because~$\|G\|_{2j} = \|g\|_{2j}$.


By
inequality~\eqref{eq:LinearizedUpper}
and the fact that~$\|F\|_{2j/(2j-1)}= \|G\|_{2j}^{2j-1}$,
\begin{equation}
\label{eq:BoundPhif}
\Phi_f(G) \le (\|G\|_{2j})^{2j}.
\end{equation}
Given an integer~$n$ in the support of~$\hat f$, let~$z^n$ be the function mapping~$\theta$ to~$e^{in\theta}$.
Another useful property of~$(\|\cdot\|_{2j})^{2j}$
is that, for a real parameter~$t$,
\begin{equation}
\label{eq:NormDerivative}
\left[\frac{d}{dt}\left(\|G + tz^n\|_{2j}\right)^{2j}\right]_{t=0}
= 2j\hat F(n).
\end{equation}
As in~\cite{HaL}, this follows from  differentiation inside an integral sign, with~$2j$ replaced by any exponent in the interval~$(1, \infty)$.
For the~$L^{2j}$ norm,
however, it is enough to expand~$(G + tz^n)^j(\overline{G + tz^n})^j$
as a polynomial in powers of~$t$, integrate each term with respect to~$\theta$, and examine the resulting coefficient of~$t$.

The conclusions of the theorem are trivial when~$\|f\|_{2j/(2j-1)} = 0$.
In the remaining cases,
as~$H$ runs through~$PD^{2j}(g,f)$,
the~$L^{2j}$ norm of~$H$ is constant,
and~$\Phi_f(H) \ne 0$.
So the fraction~$[(\|H\|_{2j})/\Phi_f(H)]^{2j}$ is minimized when~$H = G$.
That fraction extends
to the nontrivial part of~$PD^{2j}_f$,
that is
to
the union
of rays of the form
\[
\{sH: s > 0, H \in~PD^{2j}(g,f)\}.
\]
The fraction
is constant on each such ray, and minimal on the one where~$H = G$.
Form the quotient
\[
Q_n(t) := \frac{(\|G + tz^n\|_{2j})^{2j}}{\Phi_f(G + tz^n)^{2j}}
\]
when~$t \ge 0$.
The
derivative
of~$Q_n(t)$
at~$t=0$
exists and is
equal to
\[
\frac{2j\Phi_f(G)^{2j-1}}{\Phi_f(G)^{4j}}
\left[\Phi_f(G)\hat F(n) - \left(\|G\|_{2j}\right)^{2j}
\left|\hat f(n)\right|\right].
\]
Since~$Q_n(t)$ is minimal at~$t=0$, this derivative is nonnegative. So
\begin{equation}
\label{eq:OverMajorize}
\hat F(n) \ge
\frac{\left(\|G\|_{2j}\right)^{2j}}
{\Phi_f(G)}\left|\hat f(n)\right| \ge \left|\hat f(n)\right|,
\end{equation}
by inequality~\eqref{eq:BoundPhif}.

Now let~$f$ be any function in~$L^{2j/(2j-1)}$.
By the density of trigonometric polynomials in~$L^{2j/(2j-1)}$, one can write~$
f$ as an infinite sum of trigonometric polynomials~$f_k$, where
\[
\sum_k\|f_k\|_{2j/(2j-1)} \le (1 + \varepsilon)\|f\|_{2j/(2j-1)}.
\]
By inequality~\eqref{eq:LinearizedUpper},
the series~$\sum_k\Phi_{f_k}$
converges uniformly
to~$\Phi_f$
on bounded
subsets of~$PD^{2j}$.
It follows that~$\Phi_f$ is continuous
on~$PD^{2j}$
and that~$\Phi_f(H) \le \|f\|_{2j/(2j-1)}\|H\|_{2j}$
for all~$H$ in~$PD^{2j}$.

Since the function~$\Phi_f$ is bounded on the set~$PD^{2j}_{(g, f)}$,
there is 
a sequence~$(G_k)$
in~$PD^{2j}_{(g, f)}$
for
which~$\Phi_f(G_k)$ converges to the supremum of the values of~$\Phi_f$ in that set.
If~$(G_k)$ is a Cauchy sequence, then by
the continuity of~$\Phi_f$, that supremum is attained at~$G = \lim_{k\to\infty}G_k$.

By the uniform convexity of
the~$L^{2j}$ norm,
for each positive number~$\varepsilon$
there is a number~$R$ in the interval~$[0, 1)$
with the following property.
If~$H$ and~$K$ both belong to~$PD^{2j}(g, f)$, and if~$\|(H+K)/2\|_{2j} > R\|g\|_{2j}$,
then~$\|H-K\|_{2j} < \varepsilon$.
This reduces the task of proving that~$(G_k)$ is a Cauchy sequence to checking
when~$R \in [0, 1)$
that if~$k$ and~$k'$ are large enough, then~$\|(G_k + G_{k'})/2\|_{2j} > R\|g\|_{2j}$.

Let~$M$ be
the supremum of~$\Phi_f(H)$ as~$H$ runs through~$PD^{2j}_{(g, f)}$.
When~$k$ and~$k'$ are large enough,~$\Phi_f(G_k) > RM$,
and~$\Phi_f(G_{k'}) > RM$.
It follows that~$\Phi_f((G_k + G_{k'})/2) > RM$ too.  Rename~$(G_k + G_{k'})/2$ as~$G'$,
and let
\[
H = \frac{\|g\|_{2j}}{\|G'\|_{2j}}G'.
\]
This rescaling puts~$H$ in~$PD^{2j}_{(g, f)}$,
so that~$\Phi_f(H) \le M$.
On the other
hand,
\[
\Phi_f(H) = \frac{\|g\|_{2j}\Phi_f(G')}{\|G'\|_{2j}}.
\]
Therefore,
\[
RM < \Phi_f(G') = \frac{\|G'\|_{2j}}
{\|g\|_{2j}}\Phi_f(H)
\le \frac{\|G'\|_{2j}}{\|g\|_{2j}}M,
\]
and~$\|G'\|_{2j} > R\|g\|_{2j}$
as required.
\end{proof}

\begin{remark}
\label{rm:Minimizing}
Another way to describe~$G$
is that letting~$H=G$ minimizes~$\|H\|_{2j}$
in the part of~$PD^{2j}_f$ where~$\Phi_f(H) = M$.
In Appendix~\ref{Continuous} below,
this approach is
used
in~$L^{2j}(\R^n)$.

\end{remark}

\begin{remark}
\label{rm:Reciprocal}
Putting~$H = G/M$
minimizes~$\Phi_f(H)$ 
subject to the constraint that~$\Phi_f(H) = 1$.
It is shown in~\cite[\S3]{FV}
that~$\|G/M\|_{2j}$ must then be the reciprocal of the norm of the minimal
majorant,~$F'$ say,
of~$f$ in~$L^{2j/(2j-1)}$.
On the other
hand,
\begin{equation}
\label{eq:NormRatio}
\frac{\|F\|_{2j/(2j-1)}}{\|F'\|_{2j/(2j-1)}}
= \frac{(\|G\|_{2j})^{2j-1}}{M/\|G\|_{2j}}
= \frac{(\|G\|_{2j})^{2j}}{M}
=
\frac{\left(\|G\|_{2j}\right)^{2j}}
{\Phi_f(G)}.
\end{equation}
Denote the last fraction
above
by~$r$, and rewrite the first inequality in line~\eqref{eq:OverMajorize} as
the statement that~$\hat F \ge r|\hat f|$.
Then~$F/r$
also majorizes~$f$,
while~$\|F/r\|_{2j/(2j-1)} = \|F'\|_{2j/(2j-1)}$ by equation~\eqref{eq:NormRatio}.
The uniqueness of the minimal majorant
yields that~$F/r = F'$ and~$F = rF'$.
\end{remark}

\begin{remark}
\label{rm:NotMinimal}
When~$r > 1$, the majorant~$F$ does not have minimal norm in~$L^{2j/(2j-1)}$, and~$f$ has other majorants with the same norm as~$f$. 
Indeed, for
any nontrivial function,~$H$ say, in~$L^{2j/(2j-1)}$ with nonnegative coefficients, adding a suitably rescaled copy of~$H$ to~$F'$ gives a majorant for~$f$ with the same norm as~$f$.
That majorant is not
a rescaled copy of~$F'$ unless~$H$ is.

\end{remark}

\begin{remark}
\label{rm:LargerSupports}

In the description of~$G$ above, it is not necessary to insist a priori that~$\hat G$ vanish off the support of~$\hat f$.
Consider minimizing~$\|G\|_{2j}$ with the constraints that~$\hat G \ge 0$ and~$\Phi_f(G) = M$,
for instance.
If
the
solution
had some strictly positive coefficient outside the support of~$\hat f$, then replacing that coefficient with~$0$ would decrease
the
norm of the solution
in~$L^{2j}$ without changing the value of~$\Phi_f$.

\end{remark}

\section{Convolution on the integers}
\label{sec:ConvoFacts}


We reformulate
part of
Theorem~\ref{th:Maximizer}. 

\begin{theorem}
\label{th:powersofg}
Let~$j$ be a positive integer.
For each function~$g$ in~$L^{2j}$,
there is another function~$G$ in~$L^{2j}$
with the following 
properties.
\begin{enumerate}
\item\label{en:positiveg}  
$\hat G \ge 0$.

\item
\label{en:supportg}  
$\hat G = 0$ off the support
of the Fourier coefficients of~$(g\overline{g})^{j-1}g$.

\item
\label{en:dualnormg} 
$\|G\|_{2j} = \|g\|_{2j}$.

\item
\label{en:majorizeg} 
$(G\overline G)^{j-1}G$
majorizes~$(g\overline{g})^{j-1}g$.



\end{enumerate}
\end{theorem}

In particular, this holds when~$g$
is a trigonometric polynomial, and then~$G$
is a trigonometric polynomial too.
Use the term {\em light version}
of Theorem~\ref{th:Maximizer}
for 
that
case.
%
Easy approximation
arguments 
using
this
version of the theorem
yield
that every function in~$L^{2j/(2j-1)}$
has a majorant with no larger~$L^{2j/(2j-1)}$ norm.

Given a trigonometric polynomial~$g$, let~$a = \hat g$,
and let~$\tilde a$ be the sequence
with~$\tilde a(n) = \overline{a(-n)}$ for all~$n$;
these are the
coefficients of~$\overline g$.
Fix an integer~$j > 1$, and let~$(\tilde a*a)^{*j}$ denote
the~$j$-th convolution power of the convolution
product~$\tilde a*a$. 
The
Fourier coefficients
of~$|g|^{2j}$ 
are given
by the sequence~$(\tilde a*a)^{*j}$,
while those
of~$g^j(\overline g)^{j-1}$
are given by~$(a*\tilde a)^{*(j-1)}*a$.

Since~$|g|^{2j} \ge 0$,
its integral with respect to the measure~$d\theta/2\pi$
is equal to its~$0$-th Fourier coefficient,
that is to the value~$[(\tilde a*a)^{*j}](0)$
of the sequence~$(\tilde a*a)^{*j}$
at the index~$0$. Similar comments apply
to the trigonometric polynomial~$G$, with coefficients~$b$ say.
So
the
light version of Theorem~\ref{th:Maximizer}
is equivalent to
the following
statement
about convolution
on the integers.

\begin{theorem}
\label{th:light}
Let~$j$ be a positive integer.
Given a
finitely-supported
function~$a$
on the integers,
let~$c = a*(\tilde a*a)^{*(j-1)}$. Then
there
is 
a
function~$b$ on the integers with the following 
properties.
\begin{enumerate}
\item
\label{en:positivelight}  
$b \ge 0$.

\item
\label{en:supportlight}  
$b$ vanishes off the support
of~$c$. 

\item
\label{en:dualnormlight} 
$[(\tilde b*b)^{*j}](0)
= [(\tilde a*a)^{*j}](0)$.

\item
\label{en:majorizelight} 
$(b*\tilde b)^{*(j-1)}*b
\ge |c|$.

\end{enumerate}
\end{theorem}

This
follows from Theorem~\ref{th:powersofg},
and can also be proved directly by choosing
the function~$b$ to maximize
the sum~$\sum_n b(n)|c(n)|$ subject to the first three conditions enumerated above.
\begin{remark}
\label{rm:Counterparts}
The
convolution version
of equation~\eqref{eq:NormDerivative}
runs as follows.
Let~$\delta_n$ be the sequence
that takes
value~$1$ at~$n$, and that vanishes otherwise.  Given a real number~$t$,
let~$d = b + t\delta_n$,
and form~$[(\tilde d*d)^{*j}](0)$.
Then
the
derivative at~$t=0$
of the latter
is~$2j[(b*\tilde b)^{*(j-1)}*b](n)$.
It
is easy to check
this
by
expanding~$[(\tilde d*d)^{*j}](0)$
in powers of~$t$.
\end{remark}

\begin{remark}
\label{rm:ConvolutionHoelder}

In this context,
inequality~\eqref{eq:Parseval}
is only required when~$g$ is a trigonometric polynomial, and~$f$ is a suitable product of such polynomials.
The corresponding statement for convolution
is that
if~$a$ and~$k$ are
finitely supported
functions
on the
integers, then
\[
\left|\left[\tilde k*
(a*\tilde a)^{*(j-1)}*a
\right](0)\right| \le
[(\tilde k*k)^{*j}(0)
]^{1/2j}
\left[(\tilde a*a)^{*j}(0)\right]^{(2j-1)/2j}.
\]
This
follows from H\"older's inequality on the unit 
circle.
It
can
also
be verified
in the style of the proof of Theorem~\ref{th:Maximizer}.
Just rescale
in nontrivial cases
to make~$(\tilde a*a)^{*j}(0)$
equal to~$1$,
and then
minimize
\[
\frac{(\tilde k*k)^{*j}(0)}
{
\left|\left[\tilde k*
(a*\tilde a)^{*(j-1)}*a
\right](0)\right|^{2j}}.
\]
The fact that positive bounded operators on Hilbert spaces have unique positive~$j$-th roots is useful here.

\end{remark}

\begin{remark}
\label{eq:AbelianOrNot}

Theorem~\ref{th:light} and the two remarks above
extend to all discrete groups, abelian or not.
H\"older's inequality
extends to
noncommutative~$L^p$ spaces
as in~\cite{Xu}.
The
counterpart of the
instance in Remark~\ref{rm:ConvolutionHoelder}
can also be proved by the method outlined there.

\end{remark}

\begin{remark}
\label{rm:Traces}

Similar comments apply in the context of trace ideals \cite{Sim}.
\end{remark}

\begin{appendix}
\section{Majorization
on 
nondiscrete 
duals
}
\label{Continuous}

We now explain how
some of
our methods extend to the spaces~$L^p(\R^n)$,
where comparisons between transforms are made on a dual copy
of~$\R^n$, which is not discrete.
When~$1 < p \le 2$, transforms of~$L^p$ functions
on~$\R^n$
can be identified with (equivalence classes of)
functions on the dual copy.
When~$p' > 2$,
functions in~$L^{p'}(\R^n)$ have transforms in the sense
of tempered distributions, but
in many cases
those transforms
are not functions.

Duality arguments in~\cite{Ra, LS}
used summability
on
the group~$\R^n$ and its dual copy
to reduce the study of 
majorant properties 
in~$L^p$ spaces
to
instances
where~$f$ and~$\hat f$
are both functions.
Again~\cite{LS, Ra}
the upper majorant property only holds when~$p$ is infinite or even, and
the lower majorant property only holds
when~$p = 1$
or~$p'$ is an even integer.
Here, we offer a different proof that the upper majorant property implies the lower majorant property for the dual
exponent when~$p$ is even;
again,
we get
more information about the forms of
some majorants.

Recall 
that a 
distribution is said to be 
\emph{nonnegative}
if it maps each nonnegative test function
to a nonnegative number, and then the distribution
acts by integration against a nonnegative Borel measure.
When the distributions acts by integration
against a function, the distribution
is nonnegative if and only if 
that
function is nonnegative almost everywhere.

When two functions~$f$ and~$F$ both belong to~$L^p$, where~$1 \le p \le 2$, call~$F$ a 
\emph{majorant}
of~$f$ if~$\hat F \ge |\hat f|$
almost everywhere
on the dual copy of~$\R^n$.
More generally,
for two distributions~$\Phi$ and~$\Psi$, call~$\Phi$ a majorant of~$\Psi$ if the transform~$\hat\Phi$
is represented by a 
nonnegative
measure,
while~$\hat\Psi$ is
represented by a
measure with the property
that~$|\hat\Psi(S)| \le \hat\Phi(S)$
for all Borel sets~$S$.

If~$p$ is even,
then the upper majorant property holds with constant~$1$.  That is,
if~$\Phi \in L^p$,
then~$\Psi \in L^p$ and~$\|\Psi\|_p \le \|\Phi\|_p$. 
The only cases of this
needed here are those
where~$\hat \Phi$ and~$\hat \Psi$
are
represented by functions
in~$L^1\cap L^\infty$.
Then those transforms also belong to~$L^2$,
and it follows that~$\Phi$ and~$\Psi$ belong to~$L^2\cap L^\infty$.
One can then compare the~$L^{2j}$ norms of~$\Phi$ and~$\Psi$ by considering the~$j$-th convolution powers of~$\hat\Phi$ and~$\hat\Psi$ and comparing their~$L^2$ norms.

Also recall that a distribution
is said to vanish
on an open set if the distribution annihilates
every test function whose support is a compact subset
of the open set. 
Then
there is 
a largest such open set, and 
the support
of the distribution is defined to be the complement
of that largest open set.
When
the distribution acts by integration
against a function,
that
support 
is the smallest closed set outside which
the function takes the value~$0$
almost everywhere.

\begin{theorem}\label{real-th}
Let~$p = 2j/(2j-1)$ for some integer~$j>1$, 
let~$f$ be a function in~$L^p(\R^n)$,
and let~$S$ be the support of
the distribution~$\hat f$.
Then~$f$ has a majorant of the form~$G^j(\overline{G})^{j-1}$,
where~$G \in L^{p'}(\R^n)$,
and~$\hat G$ is a nonnegative distribution
whose support is included 
in~$S$.
The majorant of minimal~$L^p$ norm has this form,
and~$\|G\|_{p'} \le (\|f\|_p)^{1/(2j-1)}$ in that case.
The
support of the transform of the minimal majorant
is included in the closure of
the algebraic sum of~$j$ copies of~$S$
and~$j-1$ copies of~$-S$.
\end{theorem}

\begin{proof}
Modify
the
approach in Remark~\ref{rm:Minimizing},
replacing sums with integrals,
and replacing pointwise
inequalities 
with ones
that hold almost everywhere.
For now, only require that~$1 < p < 2$.

Given a
measurable
function~$c$ in~$L^{p'}$ on the dual copy of~$\R^n$,
let~$R(c)$ be the set of 
distributions~$w$
for which~$\hat w$ can be identified with a nonnegative, bounded, measurable function with bounded support on the dual copy of~$\R^n$,
and 
for which
\begin{equation}
\label{eq:Integral}
\int_{\R^n} |c|\hat w = 1.
\end{equation}
%
Then~$R(c)$ is nonempty
when~$c$ is nontrivial,
because
there
are
bounded
sets
of positive measure on which the values of~$|c|$ are bounded away from~$0$ and~$\infty$.
%
Suitably rescaling the indicator function of such a set gives a function~$\hat w$ satisfying condition~\eqref{eq:Integral}.

The inverse transform
of that function
belongs
to both~$L^\infty(\R^n)$ and~$L^2(\R^n)$, and  hence
to~$L^{p'}(\R^n)$.
Let~$K_p(c)$
be the infimum of~$L^{p'}$ norms of members
of the nonempty convex set~$R(c)$.
Since the~$L^{p'}$ norm on is uniformly convex, the closure of~$R(c)$ in~$L^{p'}$ has a unique element of smallest norm, which must be~$K_p(c)$.

Say that a function~$F$ in~$L^p(\R^n)$
is a {\em partial majorant} of~$\check c$
if~$\hat F \ge |c|$ almost everywhere in the set where~$c \ne 0$.  The set of partial majorants is 
convex, and it is closed in~$L^{p}$.
If this set is nonempty, then it has a unique element of minimal~$L^{p}$ norm, by uniform convexity again.

\begin{lemma}\label{continuousnorm}
Let~$1 < p < 2$, and 
let~$c \in L^{p'}(\R^n)$. If~$\|c\|_{p'} \ne 0$, 
then~$\check c$ has a partial majorant
in~$L^p$ if and only if~$K_p(c) > 0$.
In that case, the minimal~$L^p$ norm of partial majorants
of~$\check c$ is equal to~$1/K_p(c)$.
The partial majorant
of minimal norm is a rescaled copy
of~$h|h|^{p'-2}$,
where~$h$ has
minimal~$L^{p'}$ norm
in the closure
of~$R(c)$
in~$L^{p'}$.
Finally, 
the distributional support 
of the transform of~$h$
is included in the distributional support
of~$c$.
\end{lemma}

%
To prove this, start with the fact that if~$\check c$
has a partial majorant~$F$ in~$L^p$,
and if~$w \in R(c)$, then
\begin{equation}\label{extremal-continuous}
1 = \int_{\R^n} \hat w|c| 
\le \int_{\R^n} \hat w\hat F
= \int_{\R^n} \overline w F
\le \|w\|_{p'}\|F\|_p.
\end{equation}
The 
second equality above is the instance of the Parseval relation that makes~$\int_{\R^n} \overline{\widehat{h'}}\hat f
= \int_{\R^n} \overline{h'} f$
when~$f \in L^p$ 
and~$h$
is the inverse transform
of a bounded function with bounded support.
%
%

Inequality~(\eqref{extremal-continuous})
makes~$\|w\|_{p'} \ge 1/\|F\|_p$
for every~$w$ in~$R(c)$,
so that~$K_p(c) \ge 1/\|F\|_p$. 
In particular,~$\|F\|_p \ge 1$ when~$K_p(c) = 1$.
Rescale~$c$
to reduce matters to the latter case,
and then let~$h$ be as specified in the statement of the lemma.

Let~$k = h|h|^{p'-2}$.
Fix a function~$w$ in the set~$R(c)$,
and let~$\phi(t)$
be equal to~$(\|h + tw\|_{p'})^{p'}$.
This 
has derivative~$p'\int \hat k\hat w$ at~$t=0$.
The quotients~$(h + tw)/(1 + t\int |c|\hat w)$
belong to the closure of~$R(c)$ in~$L^{p'}$
when~$t \ge 0$.
Take~$p'$-th powers of the~$L^{p'}$ norms of these
quotients, and require that the derivatives
with respect to~$t$ of these powers be nonnegative
at~$t=0$. The outcome is that~$\int(\hat k - |c|)\hat w \ge 0$
for all~$w$ in~$R(c)$.

It follows that~$\hat k \ge |c|$ almost everywhere
on the set where~$c \ne 0$.
On the other hand,~$k$ has~$L^p$ norm
equal to~$1$, which
is a lower bound for the the norms
of partial 
majorants
of~$\check c$ in~$L^p$.
So~$k$ must be the partial majorant of minimal~$L^p$ norm.

By definition, all functions in the set~$R(c)$
have transforms
whose distributional supports are included in the distributional support
of~$c$.
Then~$h$
has this property too. This completes the proof of the lemma.

To
deduce
the theorem,
first check that if~$p = 2j/(2j-1)$ for some integer~$j>1$,
and~$f \in L^p(\R^n)$,
then~$K_p(\hat f) \ge 1/\|f\|_p$.
To that end,
write~$\hat f = \varepsilon|\hat f|$
for a measurable function~$\varepsilon$ with absolute-value~$1$.
Now let~$h$ be any function in~$R(|\hat f|)$, and let~$h'$ be the inverse transform
of the product~$\varepsilon h$.
Then
\[
1 = \int \hat h|\hat f|
= \int \widehat{h'}\overline{\hat f}
= \int h'\overline f.
\]
By the upper majorant property,~$h' \in L^{2j}$
with~$\|h'\|_{2j} \le \|h\|_{2j}$. 
H\"older's inequality then yields that
\[
1 \le \|h'\|_{2j}\|f\|_p
\le \|h\|_{2j}\|f\|_p.
\]
This makes~$1/\|f\|_p$ a lower bound for the~$L^{p'}$ norm
of every such function~$h$,
and hence for~$K_p(\hat f)$.
In particular,~$K_p(\hat f) > 0$, and~$f$
has a minimal partial majorant,~$F$ say, in~$L^p$.
%
To see that~$\|F\|_p \le ~\|f\|_p$,
again rescale
to the case where~$K_p(\hat f) = 1$.

In that case,
write~$h$
as a limit in~$L^{p'}$ norm of a sequence~$(h_n)$
of members of the set~$R(c)$. The transforms
of the functions~$h_n$ are nonnegative bounded functions
with bounded supports. The functions~$k_n := (h_n)^j(\overline{h_n})^{j-1}$
belong to~$L^p$, 
and 
the sequence~$(k_n)$ converges in~$L^p$ to~$k$,
so that~$\left(\widehat{k_n}\right)$ converges in~$L^{p'}$
to~$\hat k$.

The
transform of~$k_n$
is
equal to the convolution
of~$j$ copies of~$\widehat{h_n}$ with~$j-1$ copies
of~$\widehat{\overline{h_n}}$. 
It 
follows
that~$\widehat{k_n}$ is nonnegative almost everywhere.
Moreover, if a test function~$\psi$ vanishes on the closure of
the algebraic sum of~$j$ copies of~$S$ and~$j-1$ copies of~$-S$,
then~$\int k_n\psi = 0$ for all~$n$.
So~$\int k\psi = 0$,
and the support of~$k$
is
included in
the closure of
that sum of sets.
Since~$\hat k$ is nonnegative almost everywhere,
it is
a full majorant of~$\check c$.
It must be the one of minimal norm because it has minimal norm among partial majorants.
\end{proof}



\begin{remark}
\label{rm:MissingContinuity}

Suitably rescaling the minimal majorant~$G$ again gives one with the same~$L^p$ norm as~$f$.
Proving the existence of such a majorant
by the method used to prove Theorem~\ref{th:Maximizer}
seems harder here, because
it is not clear
that
a suitable
counterpart of
the function~$\Phi_f$
is
continuous.

\end{remark}

\end{appendix}

\bibliographystyle{amsplain}

\begin{thebibliography}{10}


\bibitem{Ba} Gregory F. Bachelis,
\textit{On the upper and lower majorant properties in~$L^{p}(G)$},
Quarterly J. Math.
\textbf{(2) 24} (1973), 119--128.



\bibitem{Bo} R.P. Boas Jr.,
\textit{Majorant problems for trigonometric series},
J. Analyse Math.
\textbf{10} (1962/1963), 253--271.


\bibitem{CKY} Inna Chervoneva, Vladimir Kadets and Michael Yampolsky,
\textit{On the upper majorant property},
Quaestiones Math. 
\textbf{20} (1997), 29--43.



\bibitem{DGLPQ} Myriam D\'echamps-Gondim, Fran\c coise Lust-Piquard
and Herv\'e Queff\`elec,
\textit{La propri\'et\'e du majorant dans les espaces de Banach},
C. R. Acad. Sci. Paris S\'er. I Math.
\textbf{293} (1981), 117--120.


\bibitem{Ebe}
Samuel E. Ebenstein,
\textit{$\Lambda(p)$ sets and the exact majorant property,}
Proc. Amer. Math. Soc. 
\textbf{42} (1973), 533--534.



\bibitem{FV}
John J.F. Fournier and Dean Vrecko
\textit{Minimal Fourier majorants
in~$L^p$},
to appear.



\bibitem{Gr} Ben Green,
\textit{Roth's theorem in the primes}
Ann. of Math. (2) 
\textbf{161} (2005), 1609--1636. 

\bibitem{GR} 
Ben Green and Imre Z. Rusza,
\textit{On the Hardy-Littlewood majorant problem},
Math. Proc. Camb. Phil. Soc,
\textbf{137}, (2004), 511--517.

\bibitem{HaL} G.H. Hardy and J.E. Littlewood,
\textit{Notes on the theory of
series (XIX); a problem concerning majorants of Fourier series},
Quarterly J. Math.
\textbf{6} (1935), 304--315.




\bibitem{LS} 
E.T.Y. Lee and Gen-ichir\^o Sunouchi,
\textit{On the majorant properties in~$L\sp{p}(G)$},
T\^ohoku Math. J. \textbf{(2) 31} (1979), 41--48.


\bibitem{LuP} Fran\c coise Lust-Piquard,
\textit{Les propri\'et\'es du majorant et du minorant
dans les espaces de Banach},
Bull. Sci. Math. \textbf{(2) 108} (1984), 51--71.


\bibitem{MS} Gerd Mockenhaupt and Wilhelm Schlag,
\textit{On the Hardy-Littlewood majorant problem for random sets},
J. Funct. Anal. 
\textbf{256} (2009), 1189--1237.

\bibitem{Mt} Hugh L. Montgomery,
\textit{Ten lectures on the interface between analytic number theory
and harmonic analysis.}
CBMS Regional Conference Series in Mathematics, \textbf{84}
American Mathematical Society, Providence, RI,
1994. xiv+220 pp.

\bibitem{Ob} Daniel M. Oberlin,
\textit{The majorant problem for sequence space},
Quart. J. Math. Oxford Ser. (2)
\textbf{27} (1976),  227--240.




\bibitem{Ra} Michael Rains,
\textit{On the upper majorant property
for locally compact abelian groups},
Canad. J. Math. \textbf{30} (1978), 915--925.





\bibitem{Sim} Barry Simon,
\textit{Trace ideals and their applications},
London Mathematical Society Lecture Note Series
\textbf{35}
Cambridge University Press, Cambridge-New York 1979, viii+134 pp.



\bibitem{Xu} Xu, Quanhua,
\textit{Operator spaces and noncommutative~$L_p$,
The part on noncommutative~$L_p$-spaces},
Lectures in the Summer School on
Banach spaces and Operator spaces,
Nankai University -- China, 2007.

\end{thebibliography}

\end{document}